\documentclass[11pt]{amsart}

\usepackage[all]{xy}
\usepackage{url}
\usepackage{amssymb}
\usepackage{hyperref}
\usepackage[margin=1.2in]{geometry}
\theoremstyle{plain}
\newtheorem{theorem}[subsection]{Theorem}
\newtheorem{proposition}[subsection]{Proposition}
\newtheorem{lemma}[subsection]{Lemma}

\newtheorem{claim}[subsection]{Claim}

\theoremstyle{definition}
\newtheorem{definition}[subsection]{Definition}

\newtheorem{remark}[subsection]{Remark}

\newtheorem{Theorem}[subsubsection]{Theorem}
\newtheorem{Proposition}[subsubsection]{Proposition}

\newtheorem{Corollary}[subsubsection]{Corollary}

\theoremstyle{definition}
\newtheorem{Definition}[subsubsection]{Definition}

\newtheorem{Remark}[subsubsection]{Remark}

\newtheorem{innercustomthm}{Theorem}
\newenvironment{customthm}[1]
  {\renewcommand\theinnercustomthm{#1}\innercustomthm}
  {\endinnercustomthm}

\numberwithin{equation}{subsubsection}
\numberwithin{equation}{subsubsection}
\normalfont
\usepackage[T1]{fontenc}

\usepackage{tikz-cd}

\makeatletter
\def\@tocline#1#2#3#4#5#6#7{\relax
  \ifnum #1>\c@tocdepth 
 \else
    \par \addpenalty\@secpenalty\addvspace{#2}%
    \begingroup \hyphenpenalty\@M
    \@ifempty{#4}{%
      \@tempdima\csname r@tocindent\number#1\endcsname\relax
    }{%
      \@tempdima#4\relax
    }%
    \parindent\z@ \leftskip#3\relax \advance\leftskip\@tempdima\relax
    \rightskip\@pnumwidth plus4em \parfillskip-\@pnumwidth
    #5\leavevmode\hskip-\@tempdima
      \ifcase #1
       \or\or \hskip 1em \or \hskip 2em \else \hskip 3em \fi%
      #6\nobreak\relax
    \dotfill\hbox to\@pnumwidth{\@tocpagenum{#7}}\par
    \nobreak
    \endgroup
  \fi}
\makeatother
\usepackage{hyperref}

\title{(Non)Vanishing results on local cohomology of valuation rings}
\author{Rankeya Datta}
\address{Department of Mathematics, University of Michigan, 2074 East Hall, 530 Church Street, Ann Arbor, MI 48109}
\email{rankeya@umich.edu}
\date{}                                           
\begin{document}
\maketitle

\begin{abstract}
We examine local cohomology in the setting of valuation rings. The novelty of this investigation stems from the fact that valuation rings are usually non-Noetherian, whereas local cohomology has been extensively developed mostly in a Noetherian setting. Various vanishing results on local cohomology for valuation rings of finite Krull dimension are obtained, and a uniform bound on the global dimension of such rings is established. Our investigation reveals differences in the sheaf theoretic definition of local cohomology, and the algebraic definition in terms of a limit of certain Ext functors. 
\end{abstract} 

\vspace{1mm}

\renewcommand{\baselinestretch}{0.75}\normalsize
\tableofcontents
\renewcommand{\baselinestretch}{1.0}\normalsize

\section{\textbf{Introduction}}
In this paper we study local cohomology of valuation rings. Since such rings are usually non-Noetherian, some caution is required in what one means by local cohomology. We adopt Grothendieck's definition \cite{GH67} -- the derived functors of sections of a sheaf of abelian groups on a space with support in a closed set are called \emph{local cohomology functors}. The generality of this definition often necessitates Noetherian restrictions in applications of local cohomology to algebraic geometry and commutative algebra. Indeed, local cohomology has proved to be a potent tool for understanding Noetherian schemes, and hence also Noetherian rings (see \cite{ILLMMS07} for a range of applications).  Nonetheless, there have been efforts to clarify when Noetherian hypotheses are necessary, in order to be able to apply this machinery to arbitrary schemes (for instance, Gabber-Ramero \cite{GR04} and Schenzel \cite{Sch03}).

\vspace{1mm}

In commutative algebra, local cohomology with respect to an ideal $I$ of a ring $A$ is usually defined as a limit of \textrm{Ext} functors (see \cite{HT07}, \cite{Lip02}, \cite{BrSh13}) -- more precisely as the right derived functors of the \emph{$I$-torsion functor}\footnote{This terminology is borrowed from \cite[Chapter 7]{ILLMMS07}. In \cite{Lip02}, $\Gamma_I$ is more accurately called the `$I$-power torsion functor'.}, $\Gamma_I$, where for a $A$-module $M$
$$\Gamma_I(M) = \{x \in M: \exists n \in \mathbb{N} \hspace{1mm} \textrm{such that} \hspace{1mm} I^nx = 0\}.$$
The derived functors of $\Gamma_I$ are also given the name `local cohomology' because the sheaf theoretic and algebraic definitions give isomorphic cohomology modules on Noetherian affine schemes  \cite[Exercise III.3.3]{Har77}. However, we show that such isomorphisms fail when the ring $A$ is a valuation ring (Proposition \ref{local and torsion cohomology don't agree}), affirming the need for caution in what one means by local cohomology in a non-Noetherian setting. For this reason, we henceforth call the derived functors of $\Gamma_I$ \emph{$I$-torsion cohomology}. 


\vspace{1mm}

\textbf{Results:} The main results of this paper are summarized, although, for simplicity, not always in complete generality. Most of the vanishing results are obtained for valuation rings of finite Krull dimension. Since any valuation ring of the function field of an algebraic variety over the ground field has finite Krull dimension, such rings already constitute a rich and interesting class.

In the remainder of the paper, $V$ denotes a valuation ring with maximal ideal $\frak m$. We first examine torsion cohomology of valuation rings. The behavior of the $\frak m$-torsion cohomology functors is governed by whether $\frak m$ is principal: 

\begin{customthm}{\ref{derived torsion characterization}}
Let $M$ be a $V$-module.
\begin{enumerate}
\item[(1)] If $\frak m$ is principal, then $R^i\Gamma_\frak{m}(M) = 0$ for all $i \geq 2$, and $R^1\Gamma_\frak{m}(M)$ is the cokernel of the canonical map $M \rightarrow M_f$, where $f$ is a generator of $\frak m$.
\item[(2)] If $\frak m$ is not principal, then $R^i\Gamma_\frak{m}(M) \cong \textrm{Ext}^i_V(V/\frak{m}, M)$ for all $i \geq 0$. 
\end{enumerate}
\end{customthm}

\noindent Since the functors $\textrm{Ext}^i_V(V/\frak m, \_ )$ are influenced by the projective dimension of the residue field $V/\frak m$, we examine the latter in Section \ref{finiteness of projective dim}. We show that the projective dimension of $V/\frak m$ is at most 2 when $V$ has finite Krull dimension (Theorem \ref{projective dimension of valuation rings if finite Krull dim}), which gives vanishing of $\frak m$-torsion cohomology in degrees $\geq 3$ even when $\frak m$ is \emph{not} principal (Corollary \ref{vanishing of torsion cohomology for finite dimensional valuation rings}).


\vspace{1mm}

The results of Section \ref{finiteness of projective dim} generalize -- for an arbitrary ideal $I$ of a valuation ring $V$ of finite Krull dimension, the projective dimension of $V/I$ is at most 2. As a result, the following bound on global dimension is obtained:

\begin{customthm}{\ref{bound on global dimension}}
The global dimension of a valuation ring $V$ of finite Krull dimension is $\leq 2$.
\end{customthm}


\noindent A simple consequence of the finiteness of global dimension is the vanishing of $I$-torsion cohomology in degrees $\geq 3$ (see Theorem \ref{vanishing of torsion cohomology with respect to an arbitrary ideal}). Moreover, we show in Proposition \ref{derived torsion cohomology does not equal local cohomology part 1} that $3$ is an optimal lower bound for triviality of torsion cohomology.  

\vspace{1mm}

We next examine local cohomology of sheaves on Spec($V$), proving the following:

\begin{customthm}{\ref{vanishing theorem for local cohomology for valuation rings}}
Let $X = \textrm{Spec}(V)$, $Z \subseteq X$ a closed set, and $U = X - Z$. For a sheaf of abelian groups $\mathcal{F}$ on $X$, let $H^i\Gamma_Z(\mathcal{F})$ denote the $i^{th}$ local cohomology of $\mathcal{F}$ with support in $Z$. Then
\begin{enumerate}
\item[(1)] $H^i\Gamma_Z(\mathcal{F}) \cong H^{i-1}(U, \mathcal{F}|_U)$ for all $i > 1$, and $H^1\Gamma_Z(\mathcal{F}) = \textrm{coker}(res^X_U:\mathcal{F}(X)\rightarrow \mathcal{F}(U))$.
\item[(2)] If $V$ has finite Krull dimension, $H^i\Gamma_Z(\mathcal{F}) = 0$ for all $i > 1$.
\end{enumerate}
\end{customthm}

 
\noindent Thus, local cohomology computations on the spectrum of a valuation ring reduce to computations of sheaf cohomology on open subschemes. Theorem \ref{vanishing theorem for local cohomology for valuation rings} follows from the triviality of higher sheaf cohomology of abelian sheaves on the spectrum of any local ring (Lemma \ref{sheaf cohomology vanishes for local rings}). Finiteness of Krull dimension plays an important role in Theorem \ref{vanishing theorem for local cohomology for valuation rings}(2), because in this case an open subscheme of Spec($V$) is always affine (Lemma \ref{closed and open subsets of spectrum of a valuation ring}).  We end the paper with an example of a valuation ring of infinite Krull dimension for which the `affineness' of open subschemes fails (Proposition \ref{a valuation ring of infinite Krull dimension not satisfying the vanishing of local cohomology}). Consequently, one would expect Theorem \ref{vanishing theorem for local cohomology for valuation rings} to also fail. Indeed, local cohomology no longer vanishes in degrees $> 1$ even for the structure sheaf (Proposition \ref{a valuation ring of infinite Krull dimension not satisfying the vanishing of local cohomology}(4)).


\vspace{1mm}

\textbf{Acknowledgments:} I am grateful to Karen E. Smith for graciously sharing her ideas, and for her suggestions that have significantly improved the quality of this paper. I also thank her for coming up with the name `countably exhaustive ordered abelian group' (Definition \ref{countably exhaustive definition}). Bhargav Bhatt and Greg Muller's encouragement inspired this investigation, and I had numerous enlightening discussions with Bhargav at various stages. One such discussion helped in simplifying the exposition of Section \ref{Local Cohomology of Valuation Rings}. Andreas Blass and Salman Siddiqi patiently answered some set theoretic questions I encountered in Section \ref{finiteness of projective dim}. I also thank Viswambhara Makam, Takumi Murayama, Linquan Ma for helpful discussions, and Juan Felipe P\'erez for comments on a draft of this paper. {The author was partially supported by NSF Grant DMS-1501625.}

\vspace{1mm}

\section{\textbf{Preliminaries}}

All rings are assumed to be commutative, with an identity element. By a local ring we mean a commutative ring (which is not necessarily Noetherian) with a unique maximal ideal. We denote the residue field of a local ring by $\kappa$. The symbol $\mathbb N$ denotes the positive integers. The terminologies `limit' and `colimit' are preferred over `inverse/projective limit' and `direct/injective limit'. We assume the reader is familiar with basic properties of valuations and valuation rings. Both these terms are used interchangeably in the paper. A great all round reference for valuation theory is \cite[Chapter VI]{Bou89}. Valuations are sometimes defined in different ways in the literature (additive vs. multiplicative notation), so we fix the definition we use: 

\begin{definition}
\label{definition of valuation}\cite[VI.3.1, Definition 1]{Bou89}
A \textbf{valuation} $v$ on a field $K$ with \textbf{value group} $G$ (a totally ordered abelian group) is a surjective group homomorphism
$$v: K^{\times} \twoheadrightarrow G$$
such that for all $x, y \in K^{\times}$ with $x + y \neq 0$, $v(x + y) \geq \textrm{min}\{v(x), x(y)\}$. For a field extension $K/ k$, a \textbf{valuation $v$ on $K/k$} is a valuation $v$ on $K$ such that $v(k^{\times}) = \{0\}$.
\end{definition}

Given an ordered abelian group $G$, we use $G^+$ to denote the set of elements of $G$ that are strictly bigger than the identity element $0$. For $a \in G$, we use $G^{\geq a}$ to denote the set of elements of $G$ that are $\geq a$, and similarly for $G^{\leq a}$. For elements $x, y$ in a ring $R$, we use $x|y$ to denote $x$ divides $y$. 

\vspace{1mm}

To avoid confusion, we denote the $I$-torsion cohomology functors by $R^i\Gamma_I$, and the local cohomology functors with support in a closed set $Z$ by $H^i\Gamma_Z$.

\vspace{1mm}

\section{\textbf{Torsion cohomology with respect to the maximal ideal}}
\label{torsion cohomology}

Recall that given a commutative ring $A$ and an ideal $I \subset A$, we get a covariant functor
$$\Gamma_I: Mod_A \rightarrow Mod_A$$
called the \textbf{$I-torsion$} functor, where for an $A$-module $M$,
$$\Gamma_I(M) = \{m \in M: \exists n \in \mathbb{N} \hspace{1mm} \textrm{such that} \hspace{1mm} I^nm = 0\}.$$
It is easy to see that $\Gamma_I$ is left-exact, and its right-derived functors, denoted $R^i\Gamma_I$ for $i \geq 0$, will be called the \textbf{$I$-torsion cohomology} functors \footnote{This non-standard terminology is used for reasons mentioned in the Introduction.}. One can also verify that 
$$\Gamma_I(M) \cong \textrm{colim}_{t \in \mathbb{N}} \textrm{Hom}_A(A/I^t, M),$$ 
and using the fact that cohomology commutes with filtered direct limits, it follows that for any $i \geq 0$, 
$$R^i\Gamma_I(M) \cong \textrm{colim}_{t \in \mathbb{N}} \textrm{Ext}^i_A(A/I^t, M).$$

\vspace{1mm}

In this section, we will examine the functors $R^i\Gamma_I$ when $A$ is a valuation ring $V$ and $I$ is the maximal ideal $\frak m$ of $V$. The following Lemma will be useful:

\begin{lemma}
\label{value groups with least positive element}
Let $v$ be a non-trivial valuation on a field $K$ with value group $G$. Let $V$ be the corresponding valuation ring, and $\frak m$ its maximal ideal. Then the following are equivalent:
\begin{enumerate}
\item[(1)] $\frak{m}^2 \neq \frak m$.
\item[(2)] $G^{+} := \{g \in G: g > 0\}$ has a smallest element.
\item[(3)] $\frak m$ is principal.
\end{enumerate}
\end{lemma}

\begin{proof}
The equivalence of (2) and (3) follows from the fact that the set of principal ideals of $V$ is linearly ordered by inclusion. Since $v$ is a non-trivial valuation, $\frak m$ is a non-zero ideal, and so (3) $\Rightarrow$ (1) follows from Nakayama's lemma. Thus, it suffices to show (1) $\Rightarrow$ (2). We prove the contrapositive of (1) $\Rightarrow$ (2). Suppose that $G^{+}$ does not have a smallest element. Let $x \in \frak m$ be a non-zero element. Let 
$$\alpha := v(x).$$
By our assumption on G, there exists $\beta \in G$ such that $0 < \beta < \alpha$. Similarly, there exists $\gamma \in G$ such that 
$$0 < \gamma < \textrm{min}\{\beta, \alpha- \beta\}.$$
Then $0 < 2\gamma < \beta + (\alpha - \beta) = \alpha$. Choose $y \in V$ such that $v(y) = \gamma$. Then $y \in \mathfrak m$, and  
$$v(y^2) = 2\gamma < \alpha = v(x).$$ 
Hence $y^2|x$, and so $x \in (y^2) \subset \frak m^2$. This proves $\frak m \subseteq \frak m^2$, from which it follows that $\frak m = \frak m^2$.
\end{proof}

\vspace{1mm}

The Lemma can be used to give a quick characterization of the modules $R^i\Gamma_\frak m(M)$, for a module $M$ over a valuation ring $(V, \frak m)$.

\vspace{1mm}

\begin{theorem}
\label{derived torsion characterization}
Let $V$ be a valuation ring with non-zero maximal ideal $\frak m$ and residue field $\kappa$. Let $M$ be a $V$-module.
\begin{enumerate}
\item[(1)] If $\frak m$ is principal, then $R^i\Gamma_\frak m(M) = 0$ for all $i \geq 2$, and $R^1\Gamma_\frak m(M)$ is  the cokernel of the natural map $M \rightarrow M_f$, where $f$ is a generator of $\frak m$.
\item[(2)] If $\frak m$ is not principal, then $R^i\Gamma_\frak m(M) \cong Ext^i_V(\kappa, M)$ for all $i \geq 0$. In particular, $\Gamma_\frak m(M) = \{x \in M: \frak{m}x = 0\}$, i.e., $\Gamma_\frak m(M)$ is the socle of $M$.
\end{enumerate}
\end{theorem}

\begin{proof}
We prove (2) first. Note that if $\frak m$ is not principal, then using Lemma \ref{value groups with least positive element}(1) and induction, one can show that for all $n \in \mathbb {N}$, $\frak m = \frak{m}^n$. Then for all $i \geq 0$, 
$$R^i\Gamma_\frak m(M) \cong \textrm{colim}_{t \in \mathbb{N}} \textrm{Ext}^i_V(V/\frak{m}^t, M) = \textrm{colim}_{t \in \mathbb{N}} \textrm{Ext}^i_V(\kappa, M) = \textrm{Ext}^i_V(\kappa, M).$$
Now suppose that $\frak m$ is principal. Let $\frak m = (f)$, for some $f \neq 0$. Note that for all $t \in \mathbb{N}$, $f^t$ is a non-zerodivisor on $V$, giving us a short exact sequence of $V$-modules
$$0 \rightarrow V \xrightarrow{f^t\cdot} V \rightarrow V/\frak{m}^t \rightarrow 0,$$
where the first map is left multiplication by $f^t$. For a $V$-module $M$, we then get a long exact sequence of Ext-modules
\begin{equation}
\label{les for m principal}
0 \rightarrow \textrm{Hom}_V(V/\frak{m}^t, M) \rightarrow \textrm{Hom}_V(V, M) \xrightarrow{f^t\cdot} \textrm{Hom}_V(V, M) \rightarrow \textrm{Ext}^1_V(V/\frak{m}^t,M) \rightarrow \dots
\end{equation}
The projectivity of $V$ gives us $\textrm{Ext}^i_V(V,M) = 0$ for all $i \geq 1$. As a result, for all $i \geq 2$, $\textrm{Ext}^i_V(V/\frak{m}^t,M) = 0$, and so
$$R^i\Gamma_\frak m(M) \cong \textrm{colim}_{t \in \mathbb{N}} \textrm{Ext}^i_V(V/\frak{m}^t, M) = 0.$$

Since $\textrm{Hom}_V(V, M) \cong M$, looking at the first few terms of (\ref{les for m principal}) we get the exact sequence
$$0 \rightarrow \textrm{Hom}_V(V/\frak{m}^t, M) \rightarrow M \xrightarrow{f^t\cdot} M \rightarrow \textrm{Ext}^1_V(V/\frak{m}^t,M) \rightarrow 0.$$
We then get a natural map of exact sequences

\[
\begin{tikzcd}
0 \arrow{r} 
  & \textrm{Hom}_V(V/\frak{m}^t, M)\arrow[rightarrow]{r}\arrow[rightarrow]{d}{}
  & M\arrow{r}{f^t\cdot}\arrow{d}{id_M} 
  & M\arrow{r}\arrow{d}{f\cdot}
  & \textrm{Ext}^1_V(V/\frak{m}^t, M) \arrow{r}\arrow{d} 
  & 0 \\
0 \arrow{r} 
  & \textrm{Hom}_V(V/\frak{m}^{t+1}, M) \arrow{r} 
  & M \arrow{r}[swap]{f^{t+1}\cdot} 
  & M \arrow{r}
  & \textrm{Ext}^1_V(V/\frak{m}^{t+1}, M) \arrow{r}
  & 0 
\end{tikzcd}
\]
where the left and right-most vertical maps are induced by the canonical map $V/\frak{m}^{t+1} \twoheadrightarrow V/\frak{m}^t$.

\vspace{1mm}

Taking the colimit of these exact sequences over $t \in \mathbb{N}$ and using the fact that the colimit of 
$$M \xrightarrow{f\cdot} M \xrightarrow{f\cdot} M \xrightarrow{f\cdot} \dots$$
is $M_f$, we get an exact sequence
$$0 \rightarrow \Gamma_\frak m(M) \hookrightarrow M \rightarrow M_f \rightarrow \textrm{colim}_{t \in \mathbb{N}} \textrm{Ext}^1_V(V/\frak{m}^t, M) \rightarrow 0,$$
where the map $M \rightarrow M_f$ is the canonical one. Then $R^1\Gamma_\frak m(M) \cong \textrm{colim}_{t \in \mathbb{N}} \textrm{Ext}^1_V(V/\frak{m}^t, M) = \textrm{coker}(M \rightarrow M_f)$, completing the proof of (1).
\end{proof}

\vspace{1mm}

We now mount an attack on understanding $R^i\Gamma_\frak m$, when $\frak m$ is not finitely generated. Somewhat surprisingly, we will see in Section \ref{finiteness of projective dim} that for all valuation rings of finite Krull dimension, $R^i\Gamma_\frak m(\_)$ vanishes for all $i \geq 3$. Here we deal with the cases $i = 1, 2$.

\vspace{1mm}

\begin{proposition}
\label{derived torsion cohomology does not equal local cohomology part 1}
If $V$ is a valuation ring with maximal ideal $\frak m$ that is not finitely generated (equivalently not principal), then 
\begin{enumerate}
\item[(1)] $R^1\Gamma_\frak m(\frak m) \neq 0$.
\item[(2)] There exists a $V$-module $M$ for which $R^2\Gamma_\frak m(M) \neq 0$.
\end{enumerate}
\end{proposition}

\begin{proof}
Let $\kappa$ be the residue field of $V$. For a $V$-module $M$, the modules $R^i\Gamma_\frak m(M)$ are just the modules $\textrm{Ext}^i_V(\kappa, M)$ by Theorem \ref{derived torsion characterization}(2).
The short exact sequence
$$0 \rightarrow \frak m \rightarrow V  \rightarrow \kappa \rightarrow 0$$
gives us a long exact sequence of Ext-modules
$$\dots \rightarrow \textrm{Ext}^i_V(\kappa, M) \rightarrow \textrm{Ext}^i_V(V, M) \rightarrow \textrm{Ext}^i_V(\frak m, M) \rightarrow \textrm{Ext}^{i+1}_V(\kappa, M) \rightarrow \dots \hspace{1mm} .$$

\vspace{1mm}

(1) For $i = 0$ and $M = \frak m$, this gives us an exact sequence
$$0 \rightarrow \textrm{Hom}_V(\kappa, \frak m) \rightarrow \textrm{Hom}_V(V, \frak m) \rightarrow \textrm{Hom}_V(\frak m,\frak m) \rightarrow \textrm{Ext}^1_V(\kappa, \frak m).$$
Assume for contradiction that $\textrm{Ext}^1_V(\kappa, \frak m) = 0$. Then the natural map $\textrm{Hom}_V(V, \frak m) \rightarrow \textrm{Hom}_V(\frak m, \frak m)$ induced by the inclusion $\frak m \hookrightarrow V$ is surjective. This means that $\frak m$ is a direct summand of $V$, and so $\frak m$ is projective. Kaplansky showed that any projective module over a local ring is free \cite{Kap58}. But $\frak m$ cannot be free because if an ideal of a ring is a free module then it has to be principal, whereas we picked our valuation ring so that its maximal ideal is not principal. This shows that $R^1\Gamma_\frak m(\frak m) = \textrm{Ext}^1_V(\kappa, \frak m) \neq 0$.

\vspace{1mm}
 
(2) It suffices to show that there exists some $V$-module $M$ for which $\textrm{Ext}^2_V(\kappa, M) \neq 0$. Using the long exact sequence of Ext modules obtained in the beginning of the proof, and the fact that $\textrm{Ext}^i_V(V, M) = 0$ for all $i \geq 1$, we get 
$$\textrm{Ext}^2_V(\kappa, M) \cong \textrm{Ext}^1_V(\frak m, M),$$
for all $V$-modules $M$.
If $\textrm{Ext}^2_V(\kappa, M) = 0$ for all $V$-modules $M$, then $\textrm{Ext}^1_V(\frak m, M) = 0$ for all $V$-modules $M$. This again implies $\frak m$ is projective \cite[Lemma 4.1.6]{Wei94}, which, as we saw while proving (1), is impossible. 
\end{proof}

\vspace{1mm}

\begin{remark}
For an ideal $I$ in a Noetherian ring, the functors $\Gamma_I$ and $\Gamma_{\sqrt{I}}$ coincide ($\sqrt{I}$ denotes the radical of $I$). However, this property no longer holds for ideals in valuation rings. Intuition suggests this is because the radical of a finitely generated ideal of a valuation ring need not be finitely generated. Here is a specific example. Let $V$ be a valuation ring of finite Krull dimension $d \geq 1$ such that the maximal ideal $\frak m$ is not finitely generated. For instance, $V$ could be a non-Noetherian valuation ring of dimension 1. Then Spec($V$) is a single chain of prime ideals 
$$(0) = P_0 \subsetneq P_1 \subsetneq P_2 \subsetneq \dots P_{d-1} \subsetneq P_d = \frak m.$$
Pick $f \in \frak m$ such that $f \notin P_{d-1}.$ The radical of the ideal $(f)$ is clearly $\frak m$, giving us a principal ideal whose radical is not finitely generated. Now consider the $V$-module $V/(f)$. Let $x$ denote the class of 1 in $V/(f)$. Then $x \in \Gamma_{(f)}\big{(}V/(f)\big{)}$ since the annihilator of $x$ in $V$ is $(f)$. Because $\Gamma_\frak m\big{(}V/(f)\big{)}$ consists of elements of $V/(f)$ that are annihilated by $\frak m$ (Theorem \ref{derived torsion characterization}(2)), we have $x \notin \Gamma_{\sqrt{(f)}}\big{(}V/(f)\big{)} = \Gamma_\frak m\big{(}V/(f)\big{)}$, proving that $\Gamma_{(f)}\big{(}V/(f)\big{)} \neq \Gamma_{\sqrt{(f)}}\big{(}V/(f)\big{)}$. This example will reappear in Proposition \ref{local and torsion cohomology don't agree} where we show that torsion and local cohomologies do not give isomorphic modules, even in degree 0. 
\end{remark}

\vspace{1mm}

\section{\textbf{Projective dimension of the residue field}}
\label{finiteness of projective dim}

With an eye toward understanding the higher $\frak m$-torsion cohomology modules of a valuation ring $V$ when $\frak m$ is not principal, we turn to computing the projective dimension of the residue field $\kappa$. The projective dimension of a $V$-module $M$ will be denoted 
$\textrm{pd}_V(M).$ The main result is:

\vspace{2mm}

\noindent \textbf{Theorem \ref{projective dimension of valuation rings if finite Krull dim}:} Let $V$ be a valuation ring of finite Krull dimension with residue field $\kappa$. Then $\textrm{pd}_V(\kappa) \leq 2$. Moreover, $\textrm{pd}_V(\kappa) = 1$ if and only if $\mathfrak m$ is principal.

\vspace{1mm}

The proof of this theorem will take some work, and is given in \ref{projective dimension of valuation rings if finite Krull dim}. It readily implies a vanishing result on $\frak m$-torsion cohomology when $\frak m$ is not finitely generated (Corollary \ref{vanishing of torsion cohomology for finite dimensional valuation rings}).

\vspace{1mm}

It turns out that the maximal ideal of any valuation ring of finite Krull dimension can be generated by countably many elements, and we will prove more generally that $\textrm{pd}_V(\kappa)$ is bounded above by 2 whenever the maximal ideal is countably generated (see Theorem \ref{some cases when the residue field has projective dimension 2}).

\vspace{1mm}

\subsection{\textbf{Countably exhaustive ordered abelian groups}}
As a first step, we translate the property of countable generation of the maximal ideal of a valuation ring into a statement about the value group. This translation is more illuminating and will help us identify valuation rings whose maximal ideals are countably generated. Thus, we introduce the following terminology:

\vspace{1mm}

\begin{Definition}
\label{countably exhaustive definition}
Let $G$ be an ordered abelian group. Let $G^{+} = \{g \in G: g > 0\}$. Then $G$ is \textbf{countably exhaustive} if there exists a sequence $(g_n)_{n \in \mathbb{N}}$ in $G^+$ satisfying
\begin{enumerate}
\item[(i)] $g_1 \geq g_2 \geq g_3 \geq \dots \hspace{1mm} .$
\item[(ii)] $G^+ = \bigcup_{n \in \mathbb{N}} G^{\geq g_n}.$
\end{enumerate}
\end{Definition}

\vspace{1mm}

\begin{Remark}
\label{a strictly decreasing sequence can be chosen}
If $G^+$ has a smallest element, then $G$ is clearly countably exhaustive. If $G^+$ does not have a least element, and $G$ is countably exhaustive, then one can find a a strictly decreasing sequence $(g_n)_{n \in \mathbb{N}}$ in $G^+$ satisfying axiom (ii) in the above definition.
\end{Remark}

\vspace{1mm}

We next show that the notion of a countably exhaustive ordered abelian group captures the notion of countable generation of the maximal ideal of a valuation ring.

\vspace{1mm}

\begin{Proposition}
Let $v$ be a valuation on a field $K$ with value group $G$. Then the maximal ideal $\frak m$ of the valuation ring $V$ is countably generated if and only if $G$ is countably exhaustive.
\end{Proposition}

\begin{proof}
For the backward implication, suppose we have a sequence $(g_n)_{n\in \mathbb{N}}$ in $G^+$ such that $G^+ = \bigcup_n G^{\geq g_n}$. Let $a_n \in \mathfrak m$ such that $v(a_n) = g_n$. Then $\frak m = (a_1, a_2, a_3, \dots)$. For the forward implication, we may suppose $\frak m$ is not principal as otherwise $G^+$ has a smallest element and so is countably exhaustive.  Choose a countable generating set $\{x_n: n \in \mathbb{N}\}$ of $\frak m$. Define a subsequence $(x_{n_k})_{k \in \mathbb{N}}$ of this generating set inductively as follows: Let $x_{n_1} = x_1$. Given $x_{n_k}$, pick $x_{n_{k+1}}$ to be the first $x_i$ such that $i > n_k$ and $v(x_i) < v(x_{n_k})$. Since $\frak m$ is not principal, such an $x_i$ has to exist as otherwise $\frak m$ would equal the ideal $(x_{n_k})$. Clearly $(x_{n_k})_{k \in \mathbb{N}}$ is also a generating set for $\frak m$. If $g_k := v(x_{n_k})$, then $g_1 \geq g_2 \geq g_3 \geq \dots$ and $G^+ = \bigcup_{k \in \mathbb{N}} G^{\geq g_k}$. So G is countably exhaustive.
\end{proof}

We will give examples of countably exhaustive ordered abelian groups in the next subsection (Proposition \ref{ordered subgroups of R^n are exhaustive}). We end this one by proving that one can bound the projective dimension of the residue field of any valuation ring whose value group is countably exhaustive: 

\vspace{1mm}

\begin{Theorem}
\label{some cases when the residue field has projective dimension 2}
Let $v$ be a valuation on a field $K$ with value group $G$. Let $V$ be the corresponding valuation ring with maximal ideal $\frak m$, and residue field $\kappa$. 
\begin{enumerate}
\item[(1)] $\textrm{pd}_V(\kappa) = 1$ if and only if $G^+$ has a smallest element.
\item[(2)] If $G$ is countably exhaustive, then $\textrm{pd}_V(\kappa) \leq 2$.
\end{enumerate}
\end{Theorem}

\begin{proof}
If $G$ is the trivial group, then $V$ is the field $K$, and $\kappa = V$. Hence, $\textrm{pd}_V(\kappa) = 0$. Suppose $G$ is non-trivial. Then $V$ is not a field, and in particular $\textrm{pd}_V(\kappa) \geq 1$ ($\kappa$ cannot be projective because $\kappa$ is not free).
From the exact sequence 
$$\frak m \rightarrow V \rightarrow \kappa \rightarrow 0,$$ 
we get that $\textrm{pd}_V(\kappa) = 1$ if and only if $\frak m$ is projective, and the latter happens if and only if $\frak m$ is free (again using Kaplansky's characterization of projectives over local rings), hence principal since $\frak m$ is an ideal of $V$.  But principality of $\frak m$ is equivalent to $G^+$ having a smallest element by Lemma \ref{value groups with least positive element}.  This proves (1).

\vspace{1mm}

Now assume $G^+$ does not have a smallest element. By (1), $\textrm{pd}_V(\kappa) > 1$. By Remark \ref{a strictly decreasing sequence can be chosen} we have a sequence $(a_n)_{n \in \mathbb{N}}$ of elements of $G^+$ such that
$$a_1 > a_2 > a_3 > a_4 > \dots,$$
and 
$$G^+ = \bigcup_n G^{\geq a_n}.$$
Let $x_n \in \frak m$ such that $v(x_n) = a_n$. Our choice of $(a_n)_{n \in \mathbb{N}}$ shows that 
$$\frak m = (x_1, x_2, x_3, \dots),$$
and $v(x_1) > v(x_2) > v(x_3) > \dots \hspace{1mm}.$ 
Pick the obvious surjection
$$\bigoplus_{i \in \mathbb{N}} V \twoheadrightarrow \frak m.$$
If $f_i$ denotes the $i^{th}$ standard basis vector of $\bigoplus_{i \in \mathbb{N}} V$, then the above surjection maps $f_i \mapsto x_i$. We will show that the kernel of $\bigoplus_{n \in \mathbb{N}} V \twoheadrightarrow \frak m$ is generated by the set 
$$\mathcal{S} := \Bigg{\{}f_i - \frac{x_i}{x_{i+1}}f_{i+1} : i \in \mathbb{N}\Bigg{\}}.$$
Clearly $\mathcal{S}$ is linearly independent over $V$, and $\mathcal{S} \subseteq \textrm{ker}(\bigoplus_{n \in \mathbb{N}} V \twoheadrightarrow m)$. Hence the submodule, $\langle \mathcal{S} \rangle$, generated by $\mathcal{S}$ is contained in the kernel. Observe that for all $i, n \in \mathbb{N}$, the element
\begin{equation}
\label{star}
f_i - \frac{x_i}{x_{i +n}}f_{i+n}
\end{equation}
is an element of $\langle \mathcal{S} \rangle$. This is easily seen by induction on $n$. As an illustration, for $n=2$, 
$$f_i - \frac{x_i}{x_{i+2}} f_{i+2} =  \Bigg{(} f_i - \frac{x_i}{x_{i+1}} f_{i+1} \Bigg{)} + \frac{x_{i}}{x_{i+1}}\Bigg{(} f_{i+1} -  \frac{x_{i+1}}{x_{i+2}} f_{i+2} \Bigg{)} \in \langle \mathcal{S} \rangle.$$
Now suppose $a_1f_1 + a_2f_2 + \dots + a_nf_n$ is some element in $\textrm{ker}(\bigoplus_{n \in \mathbb{N}} V \twoheadrightarrow \frak m)$, where $a_i \in V$. This means that $a_1x_1 + a_2x_2 + \dots + a_nx_n = 0$. Then 
$$x_n\Bigg{(}a_1\frac{x_1}{x_n} + a_2\frac{x_2}{x_n} + \dots + a_{n-1}\frac{x_{n-1}}{x_n} + a_n\Bigg{)} = 0.$$
Since $\frak m$ is torsion-free, solving for $a_n$ we get
$$a_n = -a_1\frac{x_1}{x_n} - a_2\frac{x_2}{x_n} - \dots - a_{n-1}\frac{x_{n-1}}{x_n},$$
and so, 
$$a_1f_1 + a_2f_2 + \dots + a_nf_n = a_1\Bigg{(}f_1 - \frac{x_1}{x_n}f_n\Bigg{)} + a_2\Bigg{(}f_2 - \frac{x_2}{x_n}f_n\Bigg{)} + \dots + a_{n-1}\Bigg{(}f_{n-1} - \frac{x_{n-1}}{x_n}f_n\Bigg{)}.$$
However, by (\ref{star}), 
$$f_1 - \frac{x_1}{x_n}f_n, \hspace{1mm} f_2 - \frac{x_2}{x_n}f_n, \hspace{1mm} \dots, \hspace{1mm} f_{n-1} - \frac{x_{n-1}}{x_n}f_n \in \langle \mathcal{S} \rangle,$$
and so, $a_1f_1 + a_2f_2 + \dots + a_nf_n \in \langle \mathcal{S} \rangle$, showing that
$$\langle \mathcal{S} \rangle = \textrm{ker}\big{(}\bigoplus_{n \in \mathbb{N}} V \twoheadrightarrow \frak m\big{)}.$$
Therefore, $\textrm{ker}(\bigoplus_{n \in \mathbb{N}} V \twoheadrightarrow \frak m)$ is a free $V$-module, and $\kappa$ has a projective resolution
$$0 \rightarrow \textrm{ker}(\bigoplus_{n \in \mathbb{N}} V \twoheadrightarrow \frak m) \rightarrow \bigoplus_{n \in \mathbb{N}} V \rightarrow V \rightarrow 0,$$
proving that its projective dimension is 2.
\end{proof}

\vspace{1mm}

\begin{Remark}
\label{remark on Osofsky}
The projective dimension of ideals of a valuation ring was the subject of investigation of a paper by B. Osofsky \cite{Os67} in which the following result was established:

\vspace{1mm}

\noindent \emph{Let $V$ be a valuation ring. Let $I$ be an ideal of $V$. Then $\textrm{pd}_V(I) = n + 1$ if and only if $I$ can be generated by set of cardinality $\aleph_n$, but not by a set of smaller cardinality for all $n \geq -1$.}

\vspace{1mm}
\noindent Theorem \ref{some cases when the residue field has projective dimension 2} is a special case of Osofsky's result when the maximal ideal $\frak m$ is generated by set of cardinality at most $\aleph_0$. Osofsky's proof requires set theoretic considerations that were avoided in our proof of the $\aleph_0$ case. We will soon see that this case is already very rich, and includes all valuation rings of finite Krull dimension (Proposition \ref{finite dim valuation rings have countably exhaustive value groups}).
\end{Remark}

\vspace{1mm}

\subsection{\textbf{Examples of countably exhaustive groups}}
\label{Examples of countably exhaustive groups}

\begin{Proposition}
\label{ordered subgroups of R^n are exhaustive}
For $n \in \mathbb{N}$, consider $\mathbb{R}^{\oplus n}$ with lexicographical ordering. If $G$ is an ordered subgroup of $\mathbb{R}^{\oplus n}$, then $G$ is countably exhaustive.
\end{Proposition}

\begin{proof}
For $i = 1, \dots, n$, we let $\pi_i: \mathbb{R}^{\oplus n} \rightarrow \mathbb{R}$ denote projection onto the $i^{th}$-coordinate. {The proof follows a recursive procedure, and uses the greatest lower bound property of the real numbers. In particular, we use the convention that if a subset of $\mathbb R$ is not bounded below, then its infimum is $-\infty$.}

\vspace{1mm}

{Let $\alpha_1$ be the greatest lower bound of $\pi_1(G^+)$. We note that $\alpha_1 \geq 0$. If $\alpha_1 \notin \pi_1(G^+)$, choose a sequence $(s_n)_n \subset G^+$ such that $\pi_1(s_1) \geq \pi_1(s_2) \geq \pi_1(s_3) \geq \dots$, and
$$\lim_{n \rightarrow \infty} \pi_1(s_n) = \alpha_1.$$
Then $s_1 \geq s_2 \geq s_3 \geq \dots$ (by definition of lexicographical order), and $G^+ = \bigcup_{n \in \mathbb N} G^{\geq s_n}$, proving countable exhaustivity.}

\vspace{1mm}

If $\alpha_1 \in \pi_1(G^+)$, choose $\omega_1 \in G^+$ such that $\pi_1(\omega_1) = \alpha_1$, and let $\alpha_2$ be the greatest lower bound of {$\pi_2(\Lambda_1)$, where 
$$\Lambda_1 := G^+ \cap G^{\leq \omega_1}.$$} 
If $\alpha_2 \notin \pi_2(\Lambda_1)$, then repeat the procedure in the previous paragraph, for $\Lambda_1$ instead of $G^+$, to get countable exhaustivity of $G$. {In other words, pick $t_1, t_2, t_3, \dots \in \Lambda_1$ such that $\pi_2(t_1) \geq \pi_2(t_2) \geq \pi_2(t_3) \geq \dots$, and 
$$\lim_{n \rightarrow \infty} \pi_2(t_n) = \alpha_2.$$ 
Note $\pi_1(t_n) = \alpha_1$, for all $n$, by definition of $\alpha_1$, and since $0 < t_n \leq \omega_1$ by choice. Thus, $t_1 \geq t_2 \geq t_3 \geq \dots$, and $G^{+} = \bigcup_{n} G^{\geq t_n}$.} 

If $\alpha_2 \in \pi_2(\Lambda_1)$, choose $\omega_2 \in \Lambda_1$ such that 
$$\pi_2(\omega_2) = \alpha _2.$$
Then $\omega_2$ also satisfies $\pi_1(\omega_2) = \alpha_1$, for the same reason as the elements $t_n$ do. Continuing as above, define $\Lambda_2 := G^+ \cap G^{\leq \omega_2}$, and $\alpha_3 := \inf \pi_3(\Lambda_2)$. Depending on whether $\alpha_3 \in \pi_3(\Lambda_2)$, we repeat the above argument.

\vspace{1mm}

This process terminates after at most $n$ steps, and one of two possibilities occur --
 
\noindent (1) {There exists a smallest $j \in \{1, \dots, n\}$ such that the infimum $\alpha_j$ of $\pi_j(\Lambda_{j-1})$ is not an element of the set.} Then repeating the argument in the second paragraph of this proof for $\Lambda_{j-1}$ instead of $G^+$, one gets countable exhaustivity of $G$. 

\noindent (2) For all $j \in \{1, \dots, n\}$, $\alpha_j \in \pi_j(\Lambda_{j-1})$, allowing us to pick $\omega \in G^+$ such that
$$\omega = (\alpha_1, \alpha_2, \dots, \alpha_n) .$$ 
But $\omega$ is then the smallest element of $G^+$, and so $G$ is trivially countably exhaustive.
\end{proof}

\vspace{1mm}

A consequence of Proposition \ref{ordered subgroups of R^n are exhaustive} is that the maximal ideal of a valuation ring of finite Krull dimension is countably generated.

\vspace{1mm}

\begin{Proposition}
\label{finite dim valuation rings have countably exhaustive value groups} 
Let $v$ be a valuation on a field $K$ such that the corresponding valuation ring $V$ has finite Krull dimension $d$. Then the value group $G$ of $v$ is order-isomorphic to a subgroup of $\mathbb{R}^{\oplus d}$, with induced lexicographical ordering. In particular, $G$ is countably exhaustive.
\end{Proposition}

\begin{proof}
One has the notion of an \emph{isolated subgroup} of an ordered abelian group in valuation theory \cite[VI.4.2, Definition 1]{Bou89}, and there is a well-known inclusion reversing bijection
$$\{\textrm{prime ideal of }V\} \longleftrightarrow \{\textrm{isolated subgroup of }G\}.$$ 
For details we refer the reader to \cite[VI $\mathsection 4.1$ and VI $\mathsection 4.3$]{Bou89}. Under this bijection, the maximal ideal $\frak m$ corresponds to the trivial subgroup, and the zero ideal corresponds to $G$. Thus, the number of non-trivial isolated subgroups of $G$, denoted $\rho(G)$, equals the number of non-maximal prime ideals of $V$. Since a valuation ring of dimension $d$ has $d$ non-maximal primes ideals,  
$$\rho(G) = d.$$
Applying \cite[Chapter 2, Proposition 2.10]{Abh59} we get that $G$ is order isomorphic to a subgroup of $\mathbb{R}^{\oplus d}$, and so also countably exhaustive by Proposition \ref{ordered subgroups of R^n are exhaustive}.
\end{proof}

\vspace{1mm}

{
\begin{Remark}
Isolated subgroups are also called convex subgroups in the literature. The number $\rho(G)$ is called the \textbf{height/rank} of the ordered abelian group $G$. Following \cite{Abh59}, we have defined $\rho(G)$ (at least for valuation rings of finite Krull dimension) to be the number of non-trivial convex subgroups of $G$. However, other sources such as \cite[VI $\mathsection 4.5$, Definition 2]{Bou89} define $\rho(G)$ to be the number of proper convex subgroups of $G$. 
\end{Remark}
}

\vspace{1mm}

As a corollary, we obtain that most valuations arising in algebraic geometry have value groups that are countably exhaustive.

\vspace{1mm}

\begin{Corollary}
\label{valuations on function fields have countably exhaustive value groups}
Fix a ground field $k$. Let $K$ be a finitely generated field extension of $k$ (such as the function field of a variety over $k$). If $v$ is a valuation on $K/k$ with value group $G$, then $G$ is countably exhaustive.
\end{Corollary}

\begin{proof}
Let $d$ be the dimension of the corresponding valuation ring (at this point $d$ could be infinite), and $\kappa$ the residue field. We have the following fundamental inequality due to Abhyankar \cite[Corollary 1]{Abh56}: 
$$d + \textrm{tr.deg} \hspace{1mm} \kappa/k \leq \textrm{tr.deg} \hspace{1mm} K/k.$$
Then $d$ is finite since $\textrm{tr.deg} \hspace{1mm} K/k$ is finite, and so we are done by Proposition \ref{finite dim valuation rings have countably exhaustive value groups}.
\end{proof}

\vspace{1mm}

The proof the theorem stated at the very beginning of this section is now a matter of putting together all the results we have obtained so far:

\vspace{1mm}

\begin{Theorem}
\label{projective dimension of valuation rings if finite Krull dim}
Let $V$ be a valuation ring of finite Krull dimension with residue field $\kappa$, and assume $V$ is not a field. Then $\textrm{pd}_V(\kappa) \leq 2$. Moreover, $\textrm{pd}_V(\kappa) = 1$ if and only if $\frak m$ is principal.
\end{Theorem}

\begin{proof}
If we consider $V$ as the valuation ring associated to a valuation on the fraction field of $V$ with value group $G$, then $G$ is countably exhaustive by Proposition \ref{finite dim valuation rings have countably exhaustive value groups}. The result now follows Theorem \ref{some cases when the residue field has projective dimension 2}.
\end{proof}

\vspace{1mm}

The bound on the projective dimension of the residue field yields:

\vspace{1mm}

\begin{Corollary}
\label{vanishing of torsion cohomology for finite dimensional valuation rings}
Let $V$ be a valuation ring of finite Krull dimension. Suppose that the maximal ideal $\frak m$ of $V$ is not principal. Then
for all $V$-modules $M$, $R^i\Gamma_\frak m(M) = 0$ for all $i \geq 3$.
\end{Corollary}

\begin{proof}
Let $\kappa$ be the residue field. Since $\frak m$ is not principal, and $R^i\Gamma_\frak m(M) \cong \textrm{Ext}^i_V(\kappa, M)$ for all $i$ (Theorem \ref{derived torsion characterization}(2)), the result follows from the bound on the projective dimension of the residue field obtained in Theorem \ref{projective dimension of valuation rings if finite Krull dim} above.
\end{proof}

\vspace{1mm}

\begin{Remark}
\label{whenever countably generated, torsion vanishes}
Theorem \ref{some cases when the residue field has projective dimension 2} tells us more generally that $R^i\Gamma_\frak m$ vanishes for all $i \geq 3$ when $\frak m$ is countably generated (equivalently the value group is countably exhaustive). In Section \ref{an interesting example} we give an example of a valuation ring of infinite Krull dimension with non-finitely generated maximal ideal whose value group is countably exhaustive. Thus, countably exhaustive ordered abelian groups also include cases where the valuation ring has infinite Krull dimension.
\end{Remark}

\vspace{1mm}

\section{\textbf{Global dimension of valuation rings and torsion cohomology}}
\label{Global dimension of valuation rings of finite Krull dimension}

Recall that the global dimension of a ring $R$, denoted gldim($R$), is the supremum of the injective dimensions of all $R$-modules. One also has
$$\textrm{gldim}(R) = \sup\{\textrm{pd}_R(R/J): J \hspace{1mm} \textrm{is an ideal of} \hspace{1mm} R\} \hspace{4mm} \textrm{\cite[Theorem 4.1.2]{Wei94}}.$$ 
Surprisingly, valuation rings of finite Krull dimension have finite global dimension:

\vspace{1mm}

\begin{theorem}
\label{bound on global dimension}
Let $V$ be a valuation ring of finite Krull dimension. Then $\operatorname{gldim}(V) \leq 2$. Moreover, $\operatorname{gldim}(V) = 1$ if and only if $V$ is a discrete valuation ring.
\end{theorem}

\vspace{1mm}

Before giving the proof, note that Theorem \ref{bound on global dimension} immediately implies the following vanishing result on torsion cohomology with respect to \emph{arbitrary} ideals, generalizing Corollary \ref{vanishing of torsion cohomology for finite dimensional valuation rings}:

\vspace{1mm}

\begin{theorem}
\label{vanishing of torsion cohomology with respect to an arbitrary ideal}
Let $I$ be an ideal of a valuation ring $V$ of finite Krull dimension, and $\Gamma_I$ be the $I$-torsion functor. Then for all $V$-modules $M$ and $i \geq 3$, $R^i\Gamma_I(M) = 0$.
\end{theorem}

\begin{proof}
Since $\textrm{gldim}(V) \leq 2$ (Theorem \ref{bound on global dimension}), the injective dimension of any $V$-module $M$ is also bounded above by 2. The vanishing of $R^i\Gamma_I(M)$, for $i \geq 3$, follows.
\end{proof}

\vspace{1mm}

For the proof of Theorem \ref{bound on global dimension}, the following Lemma will be useful. It generalizes Proposition \ref{finite dim valuation rings have countably exhaustive value groups}. The strategy of proof is similar to Proposition \ref{ordered subgroups of R^n are exhaustive}.

\vspace{1mm}

\begin{lemma}
\label{ideals of finite dim val rings are countably generated}
Let $v$ be a valuation on a field $K$ with value group $G$. Suppose the corresponding valuation ring $V$ has finite Krull dimension. If $J$ is a non-zero ideal of $V$, there exists a sequence $(x_n)_n \in J$ such that $v(x_1) \geq v(x_2) \geq \dots$, and $J = (x_1, x_2, x_3, \dots)$.
\end{lemma}

\begin{proof}
We may assume $G$ is a subgroup of $\mathbb{R}^{\oplus n}$, the latter being ordered lexicographically (Proposition \ref{finite dim valuation rings have countably exhaustive value groups}). We may also assume $J \neq V$. Consider the set
$$S:= v(J - \{0\}).$$
Note $S$ has the property that if $x \in S$, then $G^{\geq x} \subseteq S$. Replacing $G^+$ by $S$ everywhere in the proof of Proposition \ref{ordered subgroups of R^n are exhaustive}, {we see that one can choose elements $s_1 \geq s_2 \geq s_3 \geq \dots$ in $S$ such that
$$S = \bigcup_{n} G^{\geq s_n}.$$
Picking $x_n \in J$ satisfying $v(x_n) = s_n$, we get $v(x_1) \geq v(x_2) \geq v(x_3) \geq \dots$ and $J = (x_1, x_2, x_3, \dots)$.} 
\end{proof}

\vspace{1mm}

\begin{remark}
In \cite[Corollary 36]{Cou03}, Lemma 5.3 is proved, more generally, for  valuation rings $V$ such that Spec($V$) is countable. But we hope our simple proof will be of some benefit.
\end{remark}

\vspace{1mm}

\begin{proof}[\textbf{Proof of Theorem \ref{bound on global dimension}}]
We may assume $V$ is not a field (fields have global dimension 0). If the global dimension of $V$ equals $1$, then $\textrm{pd}_V(V/J) \leq 1$, for all ideals $J$ in $V$. The latter is equivalent to the projectivity of $J$, which happens only when $J$ is free of rank $\leq 1$ (any ideal of a ring which is free as a module must have rank $\leq 1$). But a free ideal of rank $\leq 1$ is principal, which shows that $V$ must be a Noetherian valuation ring, that is it is discrete. On the other hand, a discrete valuation ring is a dimension 1 regular local ring, and so has global dimension 1. This proves the second assertion of the theorem.

\vspace{1mm}

Now assume that $\textrm{gldim}(V) > 1$. Then there exists an ideal $J$ of $V$ which is not finitely generated. By Lemma \ref{ideals of finite dim val rings are countably generated} one can pick a sequence $(x_n)_n \in J$ such that $v(x_1) > v(x_2) > \dots$ and $J = (x_1, x_2, x_3, \dots)$. The argument in the proof Theorem \ref{some cases when the residue field has projective dimension 2}(2) can be repeated verbatim for $J$ instead of the maximal ideal $\frak m$ to see that $\textrm{pd}_V(V/J) = 2$. Since every ideal of $V$ is countably generated (Lemma \ref{ideals of finite dim val rings are countably generated}), $V$ has global dimension $2$.
\end{proof}

\vspace{1mm}

\begin{remark}
Theorem \ref{bound on global dimension} implies that modules over valuation rings of finite Krull dimension have finite injective dimension. Injective modules over valuation rings share many common traits with injective modules over Noetherian rings. We refer the reader to \cite{Mat59}.
\end{remark}

\vspace{1mm}

\section{\textbf{Sheaf and local cohomology of valuation rings}}
\label{Local Cohomology of Valuation Rings}

Let $X$ be a topological space. Let $Z \subseteq X$ be a closed subset, and $U = X - Z$. Let $\mathfrak{Ab}_X$ denote the category of sheaves of abelian groups on $X$, and $\mathfrak{Ab}$ the category of abelian groups. We have a covariant functor
$$\Gamma_Z: \mathfrak{Ab}_X \rightarrow \mathfrak{Ab},$$
where for a sheaf $\mathcal F$, 
$$\Gamma_Z(\mathcal F) := \textrm{ker(res}^X_U: \mathcal{F}(X) \rightarrow \mathcal{F}(U)).$$
In other words, $\Gamma_Z(\mathcal F)$ is the set of global sections of $\mathcal F$ whose support is contained in $Z$. The functor $\Gamma_Z$ is clearly left-exact, and the right derived functors of $\Gamma_Z$, denoted $H^i\Gamma_Z$, are the \emph{local cohomology functors with support in $Z$}.

\vspace{1mm}

We now specialize to the case where $X$ = Spec($V$), for a valuation ring $V$. The goal will be to prove the following result:

\begin{theorem}
\label{vanishing theorem for local cohomology for valuation rings}
Let $Z$ be a closed subset of $X = Spec(V)$, for a valuation ring $V$, and $U = X - Z$. For a sheaf of abelian groups $\mathcal F$ on $X$, we have the following:
\begin{enumerate}
\item[(1)] $H^i\Gamma_Z(\mathcal F) \cong H^{i-1}(U, \mathcal{F}|_U)$ for all $i>1$, and $H^1\Gamma_Z(\mathcal F) \cong coker\big{(}\mathcal{F}(X)\xrightarrow{res^X_U} \mathcal{F}(U)\big{)}$.
\item[(2)] If $V$ has finite Krull dimension or if $U$ is quasi-compact, then $H^i\Gamma_Z(\mathcal F) = 0$ for all $i > 1$.
\end{enumerate}
\end{theorem}

\vspace{1mm}

Theorem \ref{vanishing theorem for local cohomology for valuation rings} will follow from vanishing of higher sheaf cohomology of abelian sheaves on the spectrum of any local ring (Lemma \ref{sheaf cohomology vanishes for local rings}), and some peculiarities of the Zariski topology of the spectrum of a valuation ring. 

\vspace{2mm}

The relevant properties of the Zariski topology are recorded first:

\vspace{1mm}

\begin{lemma}
\label{closed and open subsets of spectrum of a valuation ring}
Let $V$ be a valuation ring. 
\begin{enumerate}
\item[(1)] Any non-empty closed subset of Spec(V) is irreducible.
\item[(2)] An open subset $U \subseteq \operatorname{Spec}(V)$ is quasi-compact if and only if there exists $f \in V$ such that $U = D(f)$. In particular, any affine open subscheme of Spec(V) is of the form $D(f)$, and all quasicompact opens of $\operatorname{Spec}(V)$ are affine.
\item[(3)] If $V$ has finite Krull dimension, then any open subscheme of $Spec(V)$ is affine.
\end{enumerate}
\end{lemma}

\begin{proof}
(1) follows from the fact that in a valuation ring, any radical ideal is a prime ideal or the whole ring. That a proper radical ideal $I \subsetneq V$ is a prime ideal follows easily from the fact that the prime ideals that contain $I$ are totally ordered by inclusion.

\vspace{1mm}

For (2), the `if' part is clear. On the other hand, if $U$ is a quasi-compact open subscheme of Spec($V$), then there exist $f_1, \dots, f_n \in V$ such that $U = D(f_1) \cup \dots \cup D(f_n)$. Since the open subsets of Spec($V$) are totally ordered by inclusion, $U$ must equal $D(f_i)$ for some $i$. Quasi-compactness of affine opens now gives us the second statement of (2).
\vspace{1mm}

(3) is a consequence of (2). If $V$ has finite Krull dimension, the underlying set of Spec($V$) is finite. Hence any open subscheme of Spec($V$) is quasi-compact, thus affine by (2).\end{proof}

\vspace{1mm}

\begin{remark}
Lemma \ref{closed and open subsets of spectrum of a valuation ring}(3) is false without the hypothesis that $V$ has finite Krull dimension. We construct a counter-example in Section \ref{an interesting example}.
\end{remark}

\vspace{1mm}








We now show the triviality of higher sheaf cohomology on the spectrum of any local ring.

\vspace{1mm}

\begin{lemma}
\label{sheaf cohomology vanishes for local rings}
Let $R$ be a local ring, $X = \operatorname{Spec}(R)$. Then the global sections functor on the category of sheaves of abelian groups on $X$ is exact. In particular, for any sheaf of abelian groups $\mathcal F$ on $X$, $H^i(X, \mathcal F) = 0$ for all $i > 0$.
\end{lemma}

\begin{proof}
Let $\Gamma$ be the global sections functor. Since the only open set of $X$ that contains the unique closed point is $X$ itself, the stalk of any sheaf at the closed point is the global sections of that sheaf. Since taking stalks preserves exactness, $\Gamma$ is an exact functor, and all higher sheaf cohomology groups vanish.
\end{proof}

\vspace{1mm}

We can now derive Theorem \ref{vanishing theorem for local cohomology for valuation rings}.

\vspace{1mm}

\begin{proof}[\textbf{Proof of Theorem \ref{vanishing theorem for local cohomology for valuation rings}}]
We have a well-known long exact sequence involving sheaf and local cohomology \cite[Corollary 1.9]{GH67}:
$$
0 \rightarrow \Gamma_Z(\mathcal F) \rightarrow \mathcal{F}(X) \xrightarrow{\textrm{res}^X_U} \mathcal{F}(U) \rightarrow H^1\Gamma_Z(\mathcal F) \rightarrow H^1(X, \mathcal{F}) \rightarrow H^1(U, \mathcal{F}|_U) \rightarrow H^2\Gamma_Z(\mathcal F) \rightarrow \dots
$$
Here $H^i(X, \mathcal F)$ and $H^i(U, \mathcal{F}|_U)$ stand for sheaf cohomology. Since $H^i(X,\mathcal F) = 0$ for all $i \geq 1$ by Lemma \ref{sheaf cohomology vanishes for local rings}, we get 
$$H^i\Gamma_Z(\mathcal F) \cong H^{i-1}(U, \mathcal{F}|_U)$$
for all $i > 1$. The exactness of
$$0 \rightarrow \Gamma_Z(\mathcal F) \rightarrow \mathcal{F}(X) \xrightarrow{\textrm{res}^X_U} \mathcal{F}(U) \rightarrow H^1\Gamma_Z(\mathcal F) \rightarrow 0$$ 
shows that $H^1\Gamma_Z(\mathcal F)$ is the cokernel of $\mathcal{F}(X) \xrightarrow{\textrm{res}^X_U} \mathcal{F}(U)$. This proves (1).

\vspace{1mm}

For (2), if $V$ has finite Krull dimension or if $U$ is quasi-compact, then $U$ is a distinguished affine open subscheme $D(f)$ of $X$ by Lemma \ref{closed and open subsets of spectrum of a valuation ring}. In particular, $V_f$ is also a valuation ring, and so $U$ is also the spectrum of a valuation ring. Thus, Lemma \ref{sheaf cohomology vanishes for local rings} implies $H^i(U, \mathcal{F}|_U) = 0$ for $i \geq 1$. So from (1) we get $H^i\Gamma_Z(\mathcal F) = 0$ for 
$i > 1$. 
\end{proof}



\vspace{1mm}

Let $X = \operatorname{Spec}(A)$ for a Noetherian ring $A$, $M$ an $A$-module with associated sheaf $\widetilde{M}$, $I$ an ideal, and $Z = \mathbb{V}(I)$. Then the $A$-modules $R^i\Gamma_I(M)$ are isomorphic to the $A$-modules $H^i\Gamma_Z(\widetilde{M})$ for all $i \geq 0$ \cite[Exercise III.3.3]{Har77}. However, we show that the functors $\Gamma_I$ and $\Gamma_{\mathbb{V}(I)}$ give non-isomorphic $A$-modules when $A$ is a valuation ring:
\vspace{1mm}

\begin{proposition}
\label{local and torsion cohomology don't agree}
Let $V$ be any valuation ring with maximal ideal $\frak m$ that is not finitely generated, and $Z$ be the closed point of $\operatorname{Spec}(V)$. Suppose that the punctured spectrum $\operatorname{Spec}(V) - Z$ is quasi-compact (for instance if $V$ has finite Krull dimension). Then there exists a $V$-module $M$ such that $\Gamma_\frak m(M)$ and $\Gamma_Z(\widetilde{M})$ are not isomorphic $V$-modules. 
\end{proposition}

\begin{proof}
Since $\textrm{Spec}(V) - Z$ is quasi-compact, by Lemma \ref{closed and open subsets of spectrum of a valuation ring} there exists $f \in V$ such that 
$$\textrm{Spec}(V) - Z = D(f).$$
Then $f \in \frak m$, but it is not contained in any prime ideal of Spec($V$) that is not maximal.
Let $M$ be the $V$-module $V/(f)$, and let $x \in M$ denote the class of $1 \in V$. The annihilator of $x$ is $(f)$, and since $\frak m$ is not finitely generated, $(f) \subsetneq \frak m$. Thus, $x$ is not an element of $\Gamma_\frak m(M)$, because $\Gamma_\frak m(M) = \{y \in M: \frak{m}y = 0\}$ (Theorem \ref{derived torsion characterization}(2)). However, for all prime ideals $p$ of $\textrm{Spec}(V)$ that are not maximal, we have $M_p = 0$. Considering $x$ as a global section of the associated sheaf $\widetilde{M}$, it follows that its support is contained in $Z$, that is $x \in \Gamma_Z(\widetilde{M})$. Then $\Gamma_\frak m(M)$ and $\Gamma_Z(\widetilde{M})$ cannot be isomorphic $V$-modules because every element of $\Gamma_\frak m(M)$ is annihilated by $\frak m$, whereas $x$ is not.
\end{proof}

\vspace{1mm}

\begin{remark}
Let $V$ be a valuation ring of finite Krull dimension with non-zero principal (equivalently finitely generated) maximal ideal $m$. Then $\Gamma_\frak m(M) = \Gamma_{\mathbb{V}(\frak m)}(\widetilde{M})$, and Theorem \ref{vanishing theorem for local cohomology for valuation rings}, combined with Theorem \ref{derived torsion characterization}(1) implies that $R^i\Gamma_\frak m(M)$ and $H^i\Gamma_{\mathbb{V}(\frak m)}(\widetilde{M})$ are isomorphic for all $i \geq 0$.   
\end{remark}

\vspace{1mm}

\section{\textbf{A valuation ring of infinite Krull dimension}}
\label{an interesting example}

We construct a valuation ring $V$ of \emph{infinite Krull dimension} such that 
\begin{enumerate}
\item[(a)] Spec($V$) has an open subscheme that is not affine (i.e. Lemma \ref{closed and open subsets of spectrum of a valuation ring}(3) fails if Krull dimension is not finite).
\item[(b)] There exists a sheaf of abelian groups $\mathcal F$ on Spec($V$) for which $H^i\Gamma_Z(\mathcal F)$ does not vanish for some $i \geq 2$, and the closed point $Z$ (i.e. Theorem \ref{vanishing theorem for local cohomology for valuation rings}(2) fails if Krull dimension is not finite).
\end{enumerate}
In fact, in (b) we can even choose $\mathcal F$ to be the structure sheaf. Another interesting feature of this example is that the $\frak m$-torsion cohomology functors $R^i\Gamma_\frak m$ associated to this ring vanish for $i \geq 3$. Our construction is inspired by \cite[Exercises 3.3.26 and 3.3.27]{Liu06}. 
\vspace{1mm}

For the remainder of this section, $K$ will denote the field $\mathbb{C}(X_1, X_2, X_3, \dots, X_n, \dots)$, where the $X_n$ are indeterminates for all $n \in \mathbb{N}$. Let $G := \bigoplus_{n \in \mathbb{N}} \mathbb{Z}$, ordered lexicographically. The $i^{th}$ standard $\mathbb{Z}$-basis element of $G$ will be denoted $e_i$. So in the lexicographical ordering, $e_i > e_j$ if and only if $i < j$. There exists a unique valuation $v$ on $K/\mathbb{C}$ with value group $G$ such that $v(X_i) = e_i$. Let $V$ be the corresponding valuation ring, and $\frak m$ the maximal ideal of $V$. 

\begin{proposition}
\label{a valuation ring of infinite Krull dimension not satisfying the vanishing of local cohomology}
The valuation ring $V$ constructed above satisfies the following properties:
\begin{enumerate}
\item[(1)] $V$ has infinite Krull dimension.
\item[(2)] The maximal ideal $\frak m$ is not finitely generated. In particular, $\frak m = (X_1, X_2, X_3, \dots)$.
\item[(3)] The punctured spectrum $\operatorname{Spec}(V) - \{\frak m\}$ is not quasi-compact, hence is not affine.
\item[(4)] If $Z$ is the closed point of $\operatorname{Spec}(V)$, then $H^2\Gamma_Z(\mathcal{O}_{\operatorname{Spec}(V)}) \neq 0$.
\item[(5)] As a $V$-module, the residue field $\kappa$ has projective dimension $2$. 
\item[(6)] For all $V$-modules $M$, $R^i\Gamma_\frak m(M) = 0$ for all $i \geq 3$, where $R^i\Gamma_\frak m$ is the $i^{th}$ $\frak m$-torsion cohomology functor (see Section \ref{torsion cohomology}). 

\end{enumerate}
\end{proposition}

\begin{proof}
For the rest of the proof let
$$Y := \textrm{Spec}(V); \hspace{2mm} Z = \{\frak m\}; \hspace{2mm} U := \textrm{Spec}(V) - \{\frak m\}.$$

\vspace{1mm}

(1) We have
$$v(X_1)  > v(X_2) > v(X_3) > \dots,$$
which gives us a chain of ideals

\begin{equation}
\label{infinite chain of ideals}
(X_1) \subsetneq (X_2) \subsetneq (X_3) \subsetneq \dots \hspace{1mm} .
\end{equation}

Define
\begin{equation}
\label{definition of prime ideals}
P_n := \textrm{radical of the ideal} \hspace{1mm} (X_n).
\end{equation}
Then $P_n$ is a prime ideal because the radical of a proper ideal of a valuation ring is prime (see proof of Lemma \ref{closed and open subsets of spectrum of a valuation ring}(1) for an explanation). Since every power of $X_{n+1}$ has value strictly less that the value of $X_n$, it follows that $X_{n+1}$ is an element of $P_{n+1}$, but not of $P_n$. So we get an infinite chain of prime ideals
\begin{equation}
\label{infinite chain of primes}
P_1 \subsetneq P_2 \subsetneq P_3 \subsetneq \dots,
\end{equation}
which shows that $V$ has infinite Krull dimension, proving (1).

\vspace{1mm}

Using the lex ordering on $G$, it is easy to see that
$$\frak m = (X_1, X_2, X_3, \dots).$$
Hence the maximal ideal $\mathfrak m$ cannot be finitely generated, because then $\mathfrak m$ would equal $(X_i)$ for some $X_i$, which is impossible since $X_{i+1}$ would not be in $\mathfrak m$. This proves (2).

\vspace{1mm}

As a consequence of (2), we see that $U = \textrm{Spec}(V) - \mathbb{V}(\frak m) = \bigcup_{n \in \mathbb{N}} D(X_i)$. Now from (\ref{infinite chain of ideals}) we get 
$$D(X_1) \subseteq D(X_2) \subseteq D(X_3) \subseteq \dots \hspace{1mm} .$$
Since the prime ideal $P_n$ defined in (\ref{definition of prime ideals}) is an element of $D(X_{n+1})$ but not of $D(X_n)$, the inclusions $D(X_n) \subseteq D(X_{n+1})$ are strict, that is we actually have a chain of open sets
$$D(X_1) \subsetneq D(X_2) \subsetneq D(X_3) \subsetneq \dots \hspace{1mm} .$$ 
Thus, the punctured spectrum $U$ cannot be quasi-compact, because the open cover $\{D(X_n): n \in \mathbb{N}\}$ cannot have a finite sub-cover, proving (3).

\vspace{1mm}

The proof of (4) will require some work. Using Theorem \ref{vanishing theorem for local cohomology for valuation rings}, we get $H^2\Gamma_Z(\mathcal{O}_Y) \cong  H^1(U, \mathcal{O}_Y|_U)$. Hence to prove (2), it suffices to show that $H^1(U, \mathcal{O}_Y|_U) \neq 0$.

\vspace{1mm}

Let $\widetilde{K}$ denote the constant sheaf of rational functions on $Y$. Note that $\mathcal{O}_Y$ may be identified as a subsheaf of $\widetilde{K}$, and we make this identification. We get a short exact sequence of quasi-coherent sheaves of $\mathcal{O}_Y$-modules
$$0 \rightarrow \mathcal{O}_Y \rightarrow \widetilde{K} \rightarrow \widetilde{K}/\mathcal{O}_Y \rightarrow 0.$$
Restricting to the punctured spectrum $U$ gives us a corresponding short exact sequence of quasi-coherent sheaves on $U$
$$0 \rightarrow \mathcal{O}_Y|_U \rightarrow \widetilde{K}|_U \rightarrow (\widetilde{K}/\mathcal{O}_Y)|_U \rightarrow 0.$$
This gives a corresponding long-exact sequence in cohomology whose initial terms are
\begin{equation}
\label{long exact sequence in counter example}
0 \rightarrow \mathcal{O}_Y(U) \rightarrow K \rightarrow \widetilde{K}/\mathcal{O}_Y(U) \rightarrow H^1(U, \mathcal{O}_Y|_U) \rightarrow H^1(U, \widetilde{K}|_U) \rightarrow \dots \hspace{1mm} .
\end{equation}
To prove that $H^1(U, \mathcal{O}_Y|_U) \neq 0$, it suffices to show that the map $K \rightarrow (\widetilde{K}/\mathcal{O}_Y)(U)$ is not surjective. For this we need to develop a better understanding of the $\mathcal{O}_Y(U)$-module $(\widetilde{K}/\mathcal{O}_Y)(U)$.

\vspace{1mm}

\begin{claim}
\label{claim} 
$(\widetilde{K}/\mathcal{O}_Y)(U)$ is the limit (a.k.a inverse limit) of the diagram
$$ \dots \twoheadrightarrow \frac{K}{V_{X_3}} \twoheadrightarrow \frac{K}{V_{X_2}} \twoheadrightarrow \frac{K}{V_{X_1}}.$$
\end{claim}

\vspace{1mm}

The claim is not difficult to prove, but to prevent breaking the flow we postpone it until after the proof of this proposition. Note that 
$$K/V_{X_n} = (\widetilde{K}/\mathcal{O}_Y)(D(X_n)).$$ 
It is easy to check that 
$$K \rightarrow (\widetilde{K}/\mathcal{O}_Y)(U)$$
is the unique map such that for all $n \in \mathbb{N}$,
$$K \rightarrow (\widetilde{K}/\mathcal{O}_Y)(U) \xrightarrow{res^U_{D(X_n)}} K/V_{X_n} = K \twoheadrightarrow K/V_{X_n},$$
where $K \twoheadrightarrow K/V_{X_n}$ is the usual projection. We now explicitly construct an element of the limit of
$$ \dots \twoheadrightarrow \frac{K}{V_{X_3}} \twoheadrightarrow \frac{K}{V_{X_2}} \twoheadrightarrow \frac{K}{V_{X_1}}$$
which cannot be in the image of $K$, completing the proof that $H^2\Gamma_Z(\mathcal{O}_Y) \neq 0$. 

\vspace{1mm}

For $n \geq 2$,  
$$X_1^{-1}, \dots, X_{n-1}^{-1} \notin V_{X_n},$$ 
as otherwise some power of $X_n$ would be divisible by $X_i$, for some $i < n$. At the same time $X_i^{-1}$ is an element of $V_{X_n}$, for all $i \geq n$. Define 
$$\alpha_1 := 0,$$ 
and 
$$\alpha_n \hspace{1mm} \textrm{:= class of} \hspace{1mm} X_1^{-1} + \dots + X_{n-1}^{-1} \hspace{1mm} \textrm{in} \hspace{1mm} K/V_{X_n},$$ 
for all $n \geq 2$. Then  
$$(\alpha_1, \alpha_2, \alpha_3, \dots) \in \textrm{lim}_{n \in \mathbb{N}} K/V_{X_n}.$$
Assume for contradiction that $(\alpha_1, \alpha_2, \alpha_3, \dots)$ is the image of some $\alpha \in K$. There exists $n >>0$ such that 
$$\alpha \in \mathbb{C}(X_1, \dots, X_n),$$ 
and by our assumption, 
$$\alpha- \alpha_{n+2} = \alpha - (X_1^{-1} + \dots + X_{n+1}^{-1}) \in V_{X_{n+2}}.$$
Note $\alpha - X_1^{-1} - \dots - X_n^{-1}$ is also an element of field $\mathbb{C}(X_1,\dots,X_n)$, and either $\alpha - X_1^{-1} - \dots - X_n^{-1} = 0$ or $v(\alpha - X_1^{-1} - \dots - X_n^{-1}) \neq v(X_{n+1}^{-1})$. Because $X_{n+1}^{-1}$ is not an element of $V_{X_{n+2}}$, $\alpha - X_1^{-1} - \dots - X_n^{-1}$ cannot equal $0$. Thus, 
$$v(\alpha - X_1^{-1} - \dots - X_n^{-1}) \neq v(X_{n+1}^{-1}).$$
Since $v$ is a valuation, this tells us that 
$$v\big{(}\alpha - (X_1^{-1} + \dots + X_{n+1}^{-1})\big{)} = \textrm{min}\{v(\alpha - X_1^{-1} - \dots - X_n^{-1}), v(X_{n+1}^{-1})\}.$$ 
In particular, $v(\alpha - (X_1^{-1} + \dots + X_{n+1}^{-1})) \leq v(X_{n+1}^{-1})$, and so for all $m \in \mathbb{N} \cup \{0\}$ we must have 
$$v\bigg{(}X_{n+2}^m\big{(}\alpha - (X_1^{-1} + \dots + X_{n+1}^{-1})\big{)}\bigg{)} \leq v(X_{n+2}^mX_{n+1}^{-1}) < 0.$$
This contradicts 
$$\alpha - (X_1^{-1} + \dots + X_{n+1}^{-1})$$ 
being an element of $V_{X_{n+2}}$, completing the proof of (4).

\vspace{1mm}

We can, for this example, give a nice characterization of  $H^2\Gamma_Z(\mathcal{O}_Y)$. Recall that $H^2\Gamma_Z(\mathcal{O}_Y) \cong H^1(U, \mathcal{O}_Y|_U)$, and since $H^1(U, \widetilde{K}|_U) = 0$ on account of $\widetilde{K}|_U$ being a flabby sheaf on $U$, from the exactness of (\ref{long exact sequence in counter example}) it follows that $H^1(U, \mathcal{O}_Y|_U)$ is the cokernel of the map
$$K \rightarrow (\widetilde{K}/\mathcal{O}_Y)(U).$$
Thus, $H^2\Gamma_Z(\mathcal{O}_Y) \cong \textrm{coker}(K \rightarrow (\widetilde{K}/\mathcal{O}_Y)(U))$.

\vspace{1mm}

It remains to show (5) and (6). Note that $G = \bigoplus_{n \in \mathbb{N}} \mathbb{Z}$ with the lex order is countably exhaustive (see Definition \ref{countably exhaustive definition}), because the sequence formed by the basis vectors $(e_i)_{i \in \mathbb{N}}$ satisfies 
$$e_1 > e_2 > e_3 > \dots \hspace{1mm}  \textrm{and} \hspace{1mm} G^+ = \bigcup_{i \in \mathbb{N}} G^{\geq e_i}.$$
Also, $G^+$ clearly does not have a least element. Then (5) follows from Theorem \ref{some cases when the residue field has projective dimension 2}.
For (6) one can apply the proof of Theorem \ref{vanishing of torsion cohomology for finite dimensional valuation rings} verbatim, so we omit it.
\end{proof}

\vspace{1mm}

To complete the proof of the above proposition, it remains to establish Claim \ref{claim}.

\begin{proof}
[\textbf{Proof of Claim \ref{claim}}]
Let $\mathcal{A}$ be the partially ordered set whose elements are open subsets of the form $D(f)$ contained in the punctured spectrum $U$, and where the order relation is given by inclusion. In fact, $\mathcal{A}$ is totally ordered by this relation, hence in particular also a directed set. If $D(g) \subseteq D(f) \subset U$, then we have a natural map 
$$\frac{K}{\mathcal{O}_Y(D(f))} \twoheadrightarrow \frac{K}{\mathcal{O}_Y(D(g))},$$ 
induced by the restriction map $\mathcal{O}_Y(D(f)) \hookrightarrow \mathcal{O}_Y(D(g))$. This is the data of an inverse system on $\mathcal{A}$. It is well-known that 
$$(\widetilde{K}/\mathcal{O}_Y)(U) = \textrm{lim}_{\mathcal{A}} (\widetilde{K}/\mathcal{O}_Y)(D(f)) = \textrm{lim}_{\mathcal{A}} {K}/\big{(}{\mathcal{O}_Y(D(f))}\big{)}.$$

Let $\mathcal{I}$ be the subset of $\mathcal{A}$ consisting of the open sets $D(X_n)$ for $n \in \mathbb{N}$. Recall it was shown in Proposition \ref{a valuation ring of infinite Krull dimension not satisfying the vanishing of local cohomology} that $U = \bigcup_{n \in \mathbb{N}} D(X_n)$. Then $\mathcal{I}$ is cofinal in $\mathcal{A}$. This is because if $D(f)$ is any open set contained in $U$, there has to exist an $X_i$ such that $D(f) \subseteq D(X_i)$. Otherwise, $D(X_n) \subseteq D(f)$ for all $n$ since any two open subsets of $\textrm{Spec}(V)$ are comparable, and so, $D(f) = U$, which contradicts the non-quasicompactness of $U$ (Proposition \ref{a valuation ring of infinite Krull dimension not satisfying the vanishing of local cohomology}(2)). From the cofinality of $\mathcal{I}$ it follows that
$$\textrm{lim}_{\mathcal{A}} {K}/({\mathcal{O}_Y(D(f))}) = \textrm{lim}_{\mathcal{I}} {K}/({\mathcal{O}_Y(D(X_n))}).$$
But the latter is precisely the limit of
\[ \dots \twoheadrightarrow \frac{K}{V_{X_3}} \twoheadrightarrow \frac{K}{V_{X_2}} \twoheadrightarrow \frac{K}{V_{X_1}}. \qedhere \]
\end{proof}


\vspace{3mm}

\bibliographystyle{amsalpha}


\end{document}